\documentclass{article} 
\usepackage[final, nonatbib]{neurips_2019}

\newcommand{\eg}{{\it e.g.}}
\newcommand{\ie}{{\it i.e.}}

\newcommand{\BA}{\begin{array}}
\newcommand{\EA}{\end{array}}

\newcommand{\BIT}{\begin{itemize}}
\newcommand{\EIT}{\end{itemize}}

\newcommand{\reals}{{\mathbb{R}}} 

\newcommand{\argmin}{\mathop{\rm argmin}}
\newcommand{\argmax}{\mathop{\rm argmax}}

\usepackage[utf8]{inputenc} 
\usepackage[T1]{fontenc}    
\usepackage{hyperref}       
\usepackage{url}            
\usepackage{booktabs}       
\usepackage{amsfonts}       
\usepackage{nicefrac}       
\usepackage{microtype}      

\usepackage{graphicx}
\usepackage{amsthm}
\usepackage{amsmath}
\usepackage{color}
\usepackage{subcaption}

\newcommand{\Hc}{\mathcal{H}}
\newcommand{\Lc}{\mathcal{L}}

\newcommand{\etal}{\textit{et al.}}

\newtheorem*{theorem*}{Theorem}
\newtheorem{theorem}{Theorem}
\newtheorem{lemma}{Lemma}

\newtheorem{assumption}{Assumption}

\title{Hamiltonian descent for composite objectives}

\author{
  Brendan O'Donoghue\\
  DeepMind\\
  \texttt{bodonoghue@google.com} \\
  \And
  Chris J. Maddison\\
  DeepMind / University of Oxford\\
  \texttt{cmaddis@google.com} \\
}

\begin{document}
\maketitle

\begin{abstract}
In optimization the duality gap between the primal and the dual problems is a
measure of the suboptimality of any primal-dual point. In classical mechanics
the equations of motion of a system can be derived from the Hamiltonian
function, which is a quantity that describes the total energy of the system.  In
this paper we consider a convex optimization problem consisting of the sum of
two convex functions, sometimes referred to as a composite objective, and we
identify the duality gap to be the `energy' of the system.  In the Hamiltonian
formalism the energy is conserved, so we add a contractive term to the standard
equations of motion so that this energy decreases linearly (\ie, geometrically)
with time.  This yields a continuous-time ordinary differential equation (ODE)
in the primal and dual variables which converges to zero duality gap, \ie,
optimality.  This ODE has several useful properties: it induces a natural
operator splitting; at convergence it yields both the primal and dual solutions;
and it is invariant to affine transformation despite only using first order
information.  We provide several discretizations of this ODE, some of which are
new algorithms and others correspond to known techniques, such as the
alternating direction method of multipliers (ADMM).  We conclude with some
numerical examples that show the promise of our approach. We give
an example where our technique can solve a convex quadratic minimization problem
orders of magnitude faster than several commonly-used gradient methods, including
conjugate gradient, when the conditioning of the problem is poor.  Our framework
provides new insights into previously known algorithms in the literature as well
as providing a technique to generate new primal-dual algorithms.
\end{abstract}


\section{Introduction and prior work}
In physics the Hamiltonian function represents the total energy of
a system in some set of coordinates (loosely speaking). In the most typical case
the coordinates are the position  $x \in \reals^n$ and momentum $p
\in \reals^n$, and the Hamiltonian is the sum of
the potential energy, a function of the position, and the kinetic energy, a
function of the momentum. The equations of motion for the system can be derived
from the Hamiltonian. Let us denote the Hamiltonian as $\Hc: \reals^n \times
\reals^n \rightarrow \reals$, which we assume is differentiable, then the
equations of motion  \cite{hamilton1834general} are given by
\[
\dot x_t = \nabla_p \Hc(x_t, p_t), \quad \dot p_t = -\nabla_x \Hc(x_t, p_t),
\]
where we use the notation $\dot x_t:= dx_t / dt$.  For ease of
notation we shall sometimes
use $z := (x,p) \in \reals^{2n}$ to denote
the concatenation of the position and momentum into a single quantity, in which
case we can write the Hamiltonian flow as
\begin{equation}
\label{e-flow}
\dot z_t = J \nabla \Hc(z_t), \quad J = \begin{bmatrix} 0 & I \\ -I & 0 \end{bmatrix},
\end{equation}
and note that $J^TJ = I$ and that $J$ is skew symmetric, that is $J = - J^T$,
and so $v^TJv = 0$ for any $v$.
It is easy to show that these equations of motion conserve the Hamiltonian since
$\dot \Hc(z_t) = \nabla_z \Hc(z_t)^T \dot z_t = \nabla \Hc(z_t)^T J \nabla
\Hc(z_t) = 0$.
This conservation property is required for anything that models the energy of a
system in the physical universe, but not directly useful in optimization where
the goal is convergence to an optimum. By adding a contractive term to the
Hamiltonian flow we derive an ordinary differential equation (ODE) whose
solutions converge to a minimum of the Hamiltonian. We call the resulting flow
``Hamiltonian descent''.

In optimization there has been a lot of recent interest in continuous-time
ordinary differential equations (ODEs) that when discretized yield known or
interesting novel algorithms \cite{peypouquet2009evolution, bianchi2017constant,
condat2013primal}.  In particular Su \etal \cite{su2016differential} derived a
simple ODE that corresponds to Nesterov's accelerated gradient scheme
\cite{nesterov1983method}, see also \cite{attouch2019rate}. That work was
extended in \cite{wibisono2016variational} where the authors derived a ``Bregman
Lagrangian'' framework that generates a family of continuous-time ODEs
corresponding to several discrete-time algorithms, including Nesterov's
accelerated gradient.  This was extended in \cite{wilson2019accelerating} to
derive a novel acceleration algorithm. In \cite{wilson2016lyapunov} the authors
used Lyapunov functions to analyze the convergence properties of continuous and
discrete-time systems. There is a natural Hamiltonian perspective on the Bregman
Lagrangian, which was exploited in \cite{betancourt2018symplectic} to derive
optimization methods from symplectic integrators.

In a similar vein, the authors of \cite{maddison2018hamiltonian} used a
conformal Hamiltonian system to expand the class of functions for which linear
convergence of first-order methods can be obtained by encoding information about
the convex conjugate into a kinetic energy. Follow-up work analyzed the
properties of conformal symplectic integrators for these conformal Hamiltonian
systems \cite{francca2019conformal}.

Hamiltonian mechanics have previously been applied to several areas outside of
classical mechanics \cite{rockafellar1973saddle}, most notably in Hamiltonian
Monte Carlo (HMC), where the goal is to sample from a target distribution and
Hamiltonian mechanics are used to propose moves in a Metropolis-Hastings algorithm; see
\cite{neal2011mcmc} for a good survey.  More recently Hamiltonian mechanics has
been discussed in the context of game theory \cite{balduzzi2018mechanics}, where
a symplectic gradient algorithm was developed that converges to stable fixed
points of general games.

\subsection{The convex conjugate}
The Hamiltonian as used in physics is derived by taking the Legendre transform
(or convex conjugate) of one of the terms in the Lagrangian describing the
system, which for a function $f:\reals^n \rightarrow \reals$ is defined as
\[
  f^*(p) = \sup_x(x^Tp - f(x)).
\]
The function $f^*$ is always convex, even if $f$ is not. When $f$ is closed,
proper, and convex, then $(f^*)^* = f$, and $(\partial f)^{-1} = \partial f^*$,
where $\partial f$ denotes the subdifferential of $f$, which for differentiable
functions is just the gradient, \ie, $\partial f = \nabla f$ (or more precisely
$\partial f = \{ \nabla f\}$) \cite{rockafellar1970convex}.

\section{Hamiltonian descent}
A modification to the Hamiltonian flow equation (\ref{e-flow}) yields an
ordinary differential equation whose solutions decrease the Hamiltonian
linearly:
\begin{equation}
\label{e-hd}
\dot z_t = J \nabla \Hc(z_t) + z_\star - z_t,
\end{equation}
where $z_\star \in \argmin_z \Hc(z)$.
This departs from the standard Hamiltonian flow equations by the addition of the
term involving the difference between $z_\star$ and $z_t$. One can view the
Hamiltonian descent equation as a flow in a field consisting of the sum of a
standard Hamiltonian field and the negative gradient field of function
$(1/2)\|z_t - z_\star\|_2^2$.  Solutions to this differential equation
descend the level sets of the Hamiltonian and so we refer to (\ref{e-hd}) as
\emph{Hamiltonian descent} equations. Note that this flow is different to the
dissipative flows using conformal Hamiltonian mechanics studied in
\cite{maddison2018hamiltonian, francca2019conformal}, which are also
Hamiltonian descent methods but employ a different dissipative force. We
shall show the linear convergence of solutions of (\ref{e-hd}) to a minimum of
the Hamiltonian function; first we will state a necessary assumption:
\begin{assumption}
\label{a-cvxham}
The Hamiltonian $\Hc$ together with a point $(x_\star, p_\star) = z_\star \in \arg \min_{z} \Hc(z)$ satisfy the
following:
\BIT
\item $z_\star = \arg \min_{z} \Hc(z)$ is unique,
\item $\Hc(z) \geq \Hc(z_\star) = 0$ for all $z \in \reals^{2n}$,
\item $\Hc$ is proper, closed, convex,
\item $\Hc$ is continuously differentiable.
\EIT
\end{assumption}

\begin{theorem}
\label{t-ham}
If $z_t$ is following the equations of motion in (\ref{e-hd}) where $z_\star$
and the Hamiltonian function satisfy assumption \ref{a-cvxham}, then the
Hamiltonian converges to zero linearly (\ie, geometrically). Furthermore, $z_t$
converges to $z_\star$ and $\dot z_t$ converges to zero.
\begin{proof}
Consider the time derivative of the Hamiltonian:
\begin{equation}
	\label{eq:dotz}
  \dot \Hc(z_t) =\nabla \Hc(z_t)^T \dot z_t = \nabla\Hc(z_t)^T(J \nabla \Hc(z_t) + z_\star - z_t) \leq - \Hc(z_t).
\end{equation}
since $J$ is skew-symmetric, $\Hc(z_\star) = 0$ and $\Hc$ is
convex. Gr\"onwall's inequality \cite{gronwall1919note} then implies
that
$0 \leq \Hc(z_t) \leq \Hc(z_0) \exp(-t)$
and so $\Hc(z_t) \rightarrow 0$ linearly. Consider $M = \{z \in \reals^{2n} : \nabla \Hc(z)^T(z_\star - z) = 0\}$. It is not too hard to see that $M =  \{z_\star\}$ and that $M$ is an invariant set, since $\nabla \Hc(z_{\star}')^T(z_\star - z_{\star}') \geq \Hc(z)$ by convexity. Because $\Hc$ has a unique minimum, its sublevel set are bounded. Thus, we can apply Theorem 3.4 of \cite{slotine1991applied} (Local Invariant Set Theorem) to argue that all solutions $z_t \to z_\star$. Further, we have $\nabla \Hc(z_t) \to 0$ by continuity and thus $\dot z_t \to 0$.
\end{proof}
\end{theorem}

In contrast, consider the gradient descent flow $\dot z_t=
-\nabla \Hc(z_t)$, which also converges since
\[
\dot \Hc(z_t) = \nabla \Hc(z_t)^T \dot z_t= -\|\nabla \Hc(z_t)\|_2^2 \leq 0.
\]
In this case, linear convergence is only guaranteed when $\Hc$ has some other
property, such as strong convexity, which Hamiltonian descent does not require.

It may appear that these equations of motion are unrealizable without knowledge
of a minimum of the Hamiltonian $z_\star$, which would defeat the goal of finding such a point. However, by a
judicious choice of the Hamiltonian we can \emph{cancel} the terms involving
$z_\star$, and make the system realizable.
For example, take the problem of minimizing convex $f:\reals^n \rightarrow
\reals$, and consider the following Hamiltonian
\[
  \Hc(x,p) = f(x) + f^*(p) - p^T x_\star,
\]
where $x_\star$ is any minimizer of $f$. Note that $(x_\star, 0) \in
\argmin_{(x,p)} \Hc(x,p)$. Assuming $f$ and $f^*$ are continuously differentiable and $(x_\star, 0)$ is a unique minimum of $\Hc$, then it is
readily verified that this Hamiltonian satisfies assumption \ref{a-cvxham}. So the solutions of the equations of motion will converge to a minimum of $\Hc$ linearly.  In
this case the flow is given by
\begin{align*}
  \dot x_t &= \nabla_p \Hc(x_t, p_t) + x_\star - x_t = \nabla f^*(p_t) - x_t\\
  \dot p_t &= -\nabla_x \Hc(x_t, p_t) + p_\star - p_t = -\nabla f(x_t) - p_t,
\end{align*}
since $p_\star = 0$, and note that theorem \ref{t-ham} implies that $\dot x_t
\rightarrow 0$, $\dot p_t\rightarrow 0$ and in the limit these equations reduce
to the optimality condition for the problem, namely $\nabla f(x) = 0$.  However,
this system requires the ability to evaluate $\nabla f^*$, which is as hard as
the original problem (since $x_\star = \nabla f^*(0)$). In the sequel we shall
exploit the structure of composite optimization problems to avoid this
requirement.

\subsection{Affine invariance}
The Hamiltonian descent equations of motion (\ref{e-hd}) are invariant to a set
of affine transformations. This property is very useful since it means that the
performance of an algorithm based on these equations will be much less sensitive
to the conditioning of the problem than, for example, gradient descent which
does not enjoy affine invariance.

To show this property, consider a non-singular matrix $M$ that satisfies $MJM^T
= J$ and consider the Hamiltonian in the new coordinate system,
\[
\bar \Hc(y) = \Hc(M^{-1}y),
\]
where clearly $y_\star = Mz_\star$. At time $\tau$ we have
the point $y_\tau$, and let $z_\tau = M^{-1} y_\tau$. Running Hamiltonian
descent in the transformed coordinates we obtain
\begin{align*}
\dot y_\tau &= J \nabla \bar \Hc(y_\tau) + y_\star - y_\tau\\
&=JM^{-T} \nabla \Hc(M^{-1} y_\tau ) + M z_\star - M z_\tau\\
&=M(J \nabla \Hc(z_\tau) + z_\star - z_\tau)\\
&=M \dot z_\tau.
\end{align*}
Now let $z_0 = M^{-1} y_0$, then we have
$y_{t} = y_0 + \int_{0}^t \dot y_\tau= Mz_0 + \int_{0}^t M\dot z_\tau= M z_{t}$
for all $t$, and therefore $\bar \Hc(y_t) = \Hc(M^{-1} M z_t) = \Hc(z_t)$,
\ie, the original and transformed Hamiltonians have exactly the same value
for all $t$ and thus the rate of convergence is unchanged by the transformation.
The condition on $M$ is not too onerous; for example any $M$ of
the form:
\[
M = \begin{bmatrix} R & 0 \\ 0 & R^{-T} \end{bmatrix}
\]
for nonsingular $R \in \reals^{n \times n}$ satisfies the condition.
Contrast this to vanilla gradient flow,
\[
  \dot y_\tau= -\nabla \bar \Hc(y_\tau) = -M^{-T} \nabla \Hc(M^{-1}y_\tau) = M^{-T} \dot z_\tau.
\]
Again setting $z_0 = M^{-1}y_0$ we obtain $y_{t} = y_0 + \int_{0}^t
\dot y_\tau= M z_0 + \int_{0}^t M^{-T} \dot z_\tau\neq M z_{t}$ except in
the case that $M^TM = I$, \ie, $M$ is orthogonal.

\subsection{Discretizations}
There are many possible ways to discretize the Hamiltonian descent equations,
see, \eg, \cite{hairer2006geometric}. Here we present two simple approaches and
prove their convergence under certain conditions. Later we shall show that other
discretizations correspond to already known algorithms.
\subsubsection{Implicit}
Consider the following implicit discretization of (\ref{e-hd}), for some
$\epsilon > 0$ we take
\begin{equation}
\label{e-implicit}
z^{k+1} = z^k + \epsilon(J \nabla \Hc(z^{k+1}) + z_\star - z^{k+1}).
\end{equation}
Consider the change in Hamiltonian value at iteration $k$, $\Delta_k = \Hc(z^{k+1}) - \Hc(z^k)$:
\[
  \Delta_k \leq \nabla \Hc(z^{k+1})^T(z^{k+1} - z^k) =\epsilon\nabla \Hc(z^{k+1})^T(J \nabla \Hc(z^{k+1}) + z_\star - z^{k+1}) \leq -\epsilon \Hc(z^{k+1})
\]
since $J$ is skew-symmetric, $\Hc(z_\star) = 0$ and $\Hc$ is
convex. From this we have $\Hc(z^{k}) \leq (1+\epsilon)^{-k} \Hc(z_0)$.  Thus
the implicit discretization exhibits linear convergence in discrete-time,
without restriction on the step-size $\epsilon$. However, this scheme is very
difficult to implement in practice, since it requires solving a non-linear
equation for $z^{k+1}$ at every step.

\subsubsection{Explicit}
Now consider the explicit discretization
\begin{equation}
\label{e-explicit}
z^{k+1} = z^k + \epsilon(J \nabla \Hc(z^{k}) + z_\star - z^{k}),
\end{equation}
this differs from the implicit discretization in that the right hand side
depends solely on $z^k$ rather than $z^{k+1}$, and therefore is much more
practical to implement. If we assume that the gradient
of $\Hc$ is $L$-Lipschitz, then we can show that this sequence converges and
that the Hamiltonian converges to zero like $O(1/k)$. The proof of this result
is included in the appendix. If, in addition, $\Hc$ is $\mu > 0$ strongly convex, then we
can show that the Hamiltonian converges to zero like $O(\lambda^k)$ for some
$\lambda < 1$. The proof of this result, along the explicit dependence of $\lambda$ on $L$ and $\mu$ is given in the appendix.

We must mention here that both proofs are somewhat
lacking.  For example, under the assumptions of $L$-Lipschitzness of $\nabla \Hc$ and $\mu$ strong convexity of $\Hc$, our analysis requires that the step-size $\epsilon$ depend on both $L$ and $\mu$. This is a stronger requirement than the classical gradient descent analysis. Moreover, the rate $\lambda$ scales poorly with the condition number $L/\mu$ as compared to gradient descent. This may be due to the fact that both analyses depend strongly on the values of $L$ or $\mu$, which are not invariant to affine transformation even though the equations of motion are. We suspect that a tighter analysis is possible under assumptions whose structure mirror the affine invariance structure of the dynamics.

\section{Composite optimization}
Now we come to the main problem we investigate in this paper.
Consider a convex optimization problem consisting of the sum of two
convex, closed, proper functions $h : \reals^n \rightarrow \reals$ and $g :
\reals^m \rightarrow \reals$:
\begin{equation}
\label{e-primal}
  \begin{array}{ll}
    \mbox{minimize} & f(y) := h(Ay) + g(y)
  \end{array}
\end{equation}
over variable $y \in \reals^m$, with data matrix $A \in \reals^{n \times m}$.
This problem is sometimes referred to as a composite optimization problem,
see, \eg, \cite{nesterov2013gradient}.
The dual problem is given by
\begin{equation}
\label{e-dual}
  \begin{array}{ll}
    \mbox{maximize} & d(p) := -h^*(-p) - g^*(A^Tp),
  \end{array}
\end{equation}
over $p \in \reals^n$.
We assume that $h$ and $g^*$ are both differentiable, which will help ensure that the
Hamiltonian we derive satisfies assumption \ref{a-cvxham}.  Weak duality
tells us that for any $y, p$ we have $f(y) \geq d(p)$, with equality if and only
if $y$ and $p$ are primal-dual optimal, since strong duality always holds for
this problem (under mild technical conditions \cite[\S 5.2.3]{boyd2004convex}).
We can rewrite the primal and dual problems in equality constrained form:
\begin{align}
\begin{split}
\label{e-pd}
  \begin{array}{ll}
    \mbox{minimize} & h(x) + g(y)\\
    \mbox{subject to} & x = Ay,
  \end{array}
  &
  \quad
  \begin{array}{ll}
    \mbox{maximize} & -h^*(-p) - g^*(q)\\
    \mbox{subject to} & q = A^Tp,
  \end{array}
\end{split}
\end{align}
and obtain necessary and sufficient optimality conditions in terms of all four
variables:
\begin{align}
\begin{split}
\label{e-optimality}
 \nabla g^*(q_\star) - y_\star &= 0\\
 Ay_\star - x_\star  &= 0\\
-\nabla h(x_\star) - p_\star &= 0\\
A^Tp_\star - q_\star &= 0,
\end{split}
\end{align}
the proof of which is included in the appendix.

\subsection{Duality gap as Hamiltonian}
In this section we derive a partial duality gap for problem (\ref{e-pd}) and use it
as our Hamiltonian function to derive equations of motion. Then we shall
show that in the limit the equations we derive satisfy the conditions necessary
and sufficient for optimality (\ref{e-optimality}).  We start by introducing
dual variable $p$ for the equality constraint in the primal problem
(\ref{e-pd}) to obtain $h(x) + g(y) + p^T(x - Ay)$, and taking the Legendre
transform of $g$ we get the `full' Lagrangian in terms of all four primal and
dual variables:
\[
\Lc(x,y,p,q) = h(x) - g^*(q) + y^Tq + p^T(x - Ay),
\]
which is convex-concave in $(x,y)$ and $(p,q)$. We refer to this as the full
Lagrangian, because if we maximize over $(p,q)$ we recover the primal problem in
(\ref{e-pd}) and if we minimize over $(x,y)$ we recover the dual problem in
(\ref{e-pd}).  Denote by $(y_\star, p_\star)$ any primal-dual optimal point and
let $x_\star = Ay_\star$, $q_\star = A^Tp_\star$, and $f_\star = f(y_\star) =
d(p_\star)$, then a simple calculation yields
\[
\Lc(x_\star, y_\star, p,q) \leq\max_{p,q} \Lc(x_\star, y_\star, p,q) = f_\star
= \min_{x,y} \Lc(x,y,p_\star, q_\star) \leq \Lc(x,y,p_\star, q_\star).
\]
This is due to strong duality holding for this problem.  In other words, if we substitute in
the optimal primal or dual variables into the Lagrangian, then we obtain valid lower
and upper bounds respectively. Then maximizing and minimizing these bounds
over the remaining variables yields the optimal objective value, $f_\star$. Thus, the
difference between these two functions is a partial \emph{duality gap} (though
uncomputable without knowledge of a primal-dual optimal point),
\begin{align}
\begin{split}
\label{e-dual-gap}
  \mathrm{gap}(x,q) &= \Lc(x,y,p_\star, q_\star) - \Lc(x_\star, y_\star, p, q) \\
  &= h(x) -h(x_\star) + g^*(q) -g^*(q_\star) + x^Tp_\star - q^Ty_\star\\
  &\geq 0,
\end{split}
\end{align}
with equality only when the Lagrangians are equal, \ie, are optimal. Note that
the gap only depends on $x,q$, because the effect of $y$ and $p$ is cancelled
out. This gap can also be written in terms of Bregman divergences,
where the Bregman divergence between points $u$ and
$v$ induced by a differentiable convex function $h$ is defined as
$D_h(u, v) = h(u) - h(v) - \nabla h(v)^T(u - v)$,
which is always nonnegative due the convexity of $h$. Though not a true
distance metric, it does have some useful `distance-like' properties
\cite{bregman1967relaxation, bauschke1997legendre}.  We show in the appendix
that our partial duality gap can be rewritten as
\[
\mathrm{gap}(x,q) = D_h(x,x_\star) + D_{g^*}(q, q_\star).
\]
In other words, the gap also corresponds to a `distance' between the current
iterates and their optimal values, as induced by the functions $h$ and $g^*$.
Furthermore, we show in the appendix that this partial duality gap is a lower
bound on the full duality gap, \ie,
\[
f(y) - d(p) \geq \mathrm{gap}(Ay, A^Tp).
\]
The gap is not in the form of a Hamiltonian, since the variable $x$
and $q$ are of different dimension. We can reparameterize $q = A^Tp$ or
$x = Ay$, which yields two possible Hamiltonians, one in dimension $n$ and one
in dimension $m$. The first of which is
\begin{equation}
\label{e-ham1}
  \Hc(x,p) = \mathrm{gap}(x, A^Tp) = h(x) -h(x_\star) + g^*(A^Tp)
  -g^*(A^Tp_\star) + x^Tp_\star - p^Tx_\star.
\end{equation}
Due to the assumptions on $h$ and $g^*$ we know that $\Hc$ is convex and
differentiable, and evidently $\Hc(x, p) \geq \Hc(x_\star, p_\star)
= 0$. This Hamiltonian function combined with the
equations of motion in equation~(\ref{e-hd}) yields dynamics
\begin{align}
\begin{split}
\label{e-dyn1}
  \dot x_t &= \nabla_p \Hc(x_t, p_t) + x_\star - x_t = A\nabla g^*(A^Tp_t)- x_t\\
  \dot p_t &= -\nabla_x \Hc(x_t, p_t) + p_\star - p_t = -\nabla h(x_t) - p_t.
\end{split}
\end{align}
We could rewrite these equations as
\begin{align*}
\nabla g^*(q_t) - y_t &=0\\
Ay_t - x_t &=\dot x_t\\
-\nabla h(x_t) - p_t &=\dot p_t\\
A^T p_t  - q_t &=0,
\end{align*}
If $\dot x_t \rightarrow 0$ and
$\dot p_t \rightarrow 0$, then the above equations converge to the conditions
necessary and sufficient for optimality, as given in
equation (\ref{e-optimality}). This convergence could be guaranteed by theorem~\ref{t-ham}, when $\Hc$ has a unique minimum (and thus satisfies all of assumption \ref{a-cvxham}). Still, we suspect it is possible to prove the convergence of the system without this requirement on $\Hc$'s minima.

The second Hamiltonian is given by
\begin{equation}
\label{e-ham2}
  \Hc(y,q) = \mathrm{gap}(Ay, q) = h(Ay) -h(Ay_\star) + g^*(q)
  -g^*(q_\star) + y^Tq_\star - q^Ty_\star
\end{equation}
which yields equations of motion
\begin{align}
\begin{split}
\label{e-dyn2}
  \dot y_t &= \nabla_q \Hc(y_t, q_t) + y_\star - y_t =\nabla g^*(q_t)- y_t\\
  \dot q_t &= -\nabla_y \Hc(y_t, q_t) + q_\star - q_t =-A^T\nabla h(Ay_t) - q_t,
\end{split}
\end{align}
or equivalently
\begin{align*}
\nabla g^*(q_t) - y_t &=\dot y_t\\
Ay_t - x_t &=0\\
-\nabla h(x_t) - p_t &=0\\
A^T p_t  - q_t &=\dot q_t.
\end{align*}
Again, if $\dot y_t \rightarrow 0$ and
$\dot q_t \rightarrow 0$, this system will also satisfy the optimality conditions of~(\ref{e-optimality}).
Finally, theorem~\ref{t-ham} implies that both of these ODEs exhibit linear
convergence of the Hamiltonian, \ie, linear convergence of the partial duality gap
(\ref{e-dual-gap}), to zero.

\section{Connection to other methods}

\subsection{ADMM}
In this section we show how a particular discretization of our ODE yields the
well-known Alternating direction method of multipliers algorithm (ADMM)
\cite{boyd2011distributed, heyuan2012} when applied to problem (\ref{e-primal}).
We should note that in related work the authors of \cite{franca2018admm} derive
a different ODE that when discretized also yields ADMM, as well as a related ODE
that corresponds to accelerated ADMM \cite{goldstein2014fast}.  There is no
contradiction here since many ODEs can correspond to the same procedure when
discretized.

In order to prove that ADMM is equivalent to a discretization of Hamiltonian
descent we will require the generalized Moreau decomposition, which we present
next. In the statement of the lemma we use the notation $(A \partial
f A^T)$ to represent the multi-valued operator defined as $(A \partial
f A^T)(x) = A (\partial f (A^Tx)) = \{Az \mid z \in \partial f(A^Tx) \}$.
{
\begin{lemma} For convex, closed, proper function $f : \reals^m
\rightarrow \reals$ and matrix $A\in\reals^{n \times m}$, any point
$x \in \reals^n$ satisfies
\[
x = (I + \rho A \partial f A^T)^{-1} x + \rho A(\partial f^* + \rho A^TA)^{-1} A^Tx.
\]
\end{lemma}
\addtocounter{lemma}{-1}
}
We defer the proof to the appendix.
To derive ADMM we employ a standard trick in
discretizing differential equations: We add and subtract a term to the dynamics
which we shall discretize at different points, which in the limit of
infinitesimal step size will vanish, recovering the original ODE\@. Starting
from equation (\ref{e-dyn1}) and for any $\rho > 0$ the modified ODE is
\begin{align*}
  \dot p_t &= -\nabla h(x_t) - p_t - \rho(x_t - x_t)\\
  \dot x_t &= A\nabla g^*(A^T p_t) - x_t + (1/\rho)(p_t - p_t).
\end{align*}
Now we discretize as follows:
\begin{align*}
  (p^{k} - p^{k-1}) / \epsilon &= -\nabla h(x^{k+1}) -p^k - \rho(x^{k+1} - x^{k})\\
  (x^{k+1} - x^k) / \epsilon &= A\nabla g^*(A^T p^{k+1}) - x^k + (1/\rho)(p^{k+1} - p^k).
\end{align*}
Setting $\epsilon = 1$ yields
\begin{align*}
  x^{k+1} &= (\rho I + \nabla h)^{-1}(\rho x^{k} - 2 p^{k} + p^{k-1}) \\
  p^{k+1} &= (I + \rho A \nabla g^* A^T)^{-1}(p^k + \rho x^{k+1})\\
  &= p^k + \rho x^{k+1} - \rho A(\partial g + \rho A^TA)^{-1}A^T(p^k + \rho x^{k+1})\\
  &= p^k + \rho x^{k+1} - \rho A y^{k+1}
\end{align*}
where we used the generalized Moreau decomposition and introduced variable
sequence $y^k \in \reals^m$, and note that from the last equation we have that
$\rho x^{k} - 2 p^{k} + p^{k-1} = \rho Ay^{k} - p^{k}$. Finally this brings us
to ADMM; from any initial $y^0, p^0$ iterate
\begin{align*}
  x^{k+1} &= (\rho I + \nabla h)^{-1}(\rho Ay^{k} - p^{k})\\
  y^{k+1} &\in (\rho A^TA + \partial g)^{-1}A^T(p^k + \rho x^{k+1}) \\
  p^{k+1} &= p^k + \rho(x^{k+1} - Ay^{k+1}).
\end{align*}
Evidently we have lost the affine invariance property of our ODE. However we might
expect ADMM to be somewhat more robust to conditioning than gradient descent,
which appears to be true empirically \cite{boyd2011distributed}.

\subsection{PDHG}
The primal-dual hybrid gradient technique (PDHG), also called Chambolle-Pock, is
another operator splitting technique with a slightly different form to ADMM. In
particular, PDHG only requires multiplies with $A$ and $A^T$ rather than
requiring $A$ in the proximal step \cite{zhu2008efficient, esser2009general, chambolle2011first}.
When applied to problem (\ref{e-primal}) PDHG yields the following iterates
\begin{align*}
p^{k+1} &=-(I + \rho \partial h^*)^{-1}(\rho Ay^{k} - p^k)\\
y^{k+1} &= (I + \sigma \partial g)^{-1}(\sigma A^Tp^{k+1} + y^k).
\end{align*}
In the appendix we show that this corresponds to a particular discretization
of Hamiltonian descent, with step size $\epsilon = 1$.
Note that the sign of the dual variable $p^k$ is different when
compared to \cite{chambolle2011first}, this is due to the fact that the dual
problem they consider negates the dual variable when compared to ours, so this
is fixed by rewriting the iterations in terms of $-p^k$.

\section{Numerical experiments}
In this section we present two numerical examples where we compare the explicit
discretization of Hamiltonian descent flow to gradient descent. Due to the
affine invariance property of Hamiltonian descent we expect our technique to
outperform when the conditioning of the problem is poor, so we generate examples
with bad conditioning to test that.
\subsection{Regularized least-squares}
Consider the following $\ell_2$-regularized least-squares problem
\begin{equation}
\label{e-ls}
\begin{array}{ll}
\mbox{minimize} & (1/2)\|Ay - b\|_2^2 + (\lambda/2) \|By\|_2^2,
\end{array}
\end{equation}
over variable $y \in \reals^m$, where $A \in \reals^{n \times m}$, $B \in
\reals^{m \times m}$, and $\lambda \geq 0$ are data.
In the notation of problem (\ref{e-primal}) we let $h(x) = (1/2)\|x - b\|_2^2$
and $g(y) = \lambda \|By\|_2^2$, and so $\nabla g^*(q) = \argmax_y( y^T
q - \lambda \|B y\|_2^2$) which we assume is always well-defined (\ie, $B^TB$ is
invertible).  We apply the explicit discretization (\ref{e-explicit}) of the
dynamics given in equation (\ref{e-dyn2}) to this problem.  To demonstrate the
practical effect of affine invariance, we randomly generate a nonsingular matrix
$M$ and solve a sequence of optimization problems where $A$ is replaced with
$\hat A_j = AM^j$ and $B$ is replaced with $\hat B_j = BM^j$ for
$j=0,1,\ldots,j^\mathrm{max}$. Note that the optimal objective value of this
perturbed problem is unchanged from the original, and the solution for each
perturbed problem can be obtained by $(\hat y_\star)_j  = M^{-j} y_\star$,
where $y_\star$ solves the original problem (\ie, with $j=0$).  However, the
conditioning of the problem is changed - $M$ is selected so that the
conditioning of the data is worsening for increasing $j$.  We compare our
algorithm to vanilla gradient descent, to proximal gradient descent
\cite{parikh2014proximal} (where the prox-step is on the $g$ term so it is of
a similar cost to our method), and to restarted accelerated gradient descent
\cite{nesterov1983method, o2015adaptive}, and observe the effect of the
worsening conditioning.

We chose $n = 1000$ and for simplicity we chose $B=I$, $\lambda = 1$, and randomly
generated each entry in $A$ to be IID $\mathcal{N}(0, 1)$.  The best step size
was chosen via exhaustive search for all three algorithms.  The matrix $M$ was
randomly generated but chosen in such a way so as to be close to the identity.
For $j=0$ the condition number of the matrix $\hat A_j^T \hat A_j + \lambda \hat
B_j^T \hat B_j$ was $4.0 \times 10^3$, and for $j=j^\mathrm{max} = 20$ the condition number
had grown to $2.2 \times 10^{14}$, a dramatic increase.  Figure
(\ref{f-quad_primal}) shows the performance of both our technique and gradient
descent on this sequence of problems. The gradient descent traces are in orange,
with a different trace for each $j$. The fastest converging trace corresponds to
$j=0$, the best conditioned problem. As the conditioning deteriorates the
convergence is impacted, getting slower with each increase in $j$. In the
appendix we additionally include Figure~\ref{f-other_grads} which compares our
technique to proximal gradient, restarted accelerated gradient, and conjugate gradient.
All three additional
techniques display the same deterioration as the conditioning worsens.  By
problem $j=20$ no variant of gradient descent or conjugate gradient has reduced
the primal objective error, defined as $\min_k(f(y^k) - f_\star)$, to under
$O(100)$.  By contrast, our technique is completely \emph{unaffected} by the
changing data, with every trace essentially identical (up to some numerical
tolerances).  Furthermore, we used the exact same step size for every run of our
method.  This is because the discretization procedure preserved the affine
invariance of the continuous ODE it is approximating, so the changing
conditioning of the data has no effect.  In Figure (\ref{f-ham-gap}) we plot the
Hamiltonian (\ref{e-ham2}) (\ie, the partial duality gap) and the full duality
gap: $f(y^k) - d(p^k)$, for Hamiltonian descent for each value of $j$.  Once
again the traces lie directly on top of each other, until numerical errors start
to have an impact.  We note that the Hamiltonian decreases at each iteration,
and converges linearly. The duality gap and the objective values do not
necessarily decrease at each iteration, but do appear to enjoy linear
convergence for each $j$.

\begin{figure}[ht]
\centering
\begin{subfigure}{0.47\textwidth}
\centering
\includegraphics[width=\linewidth]{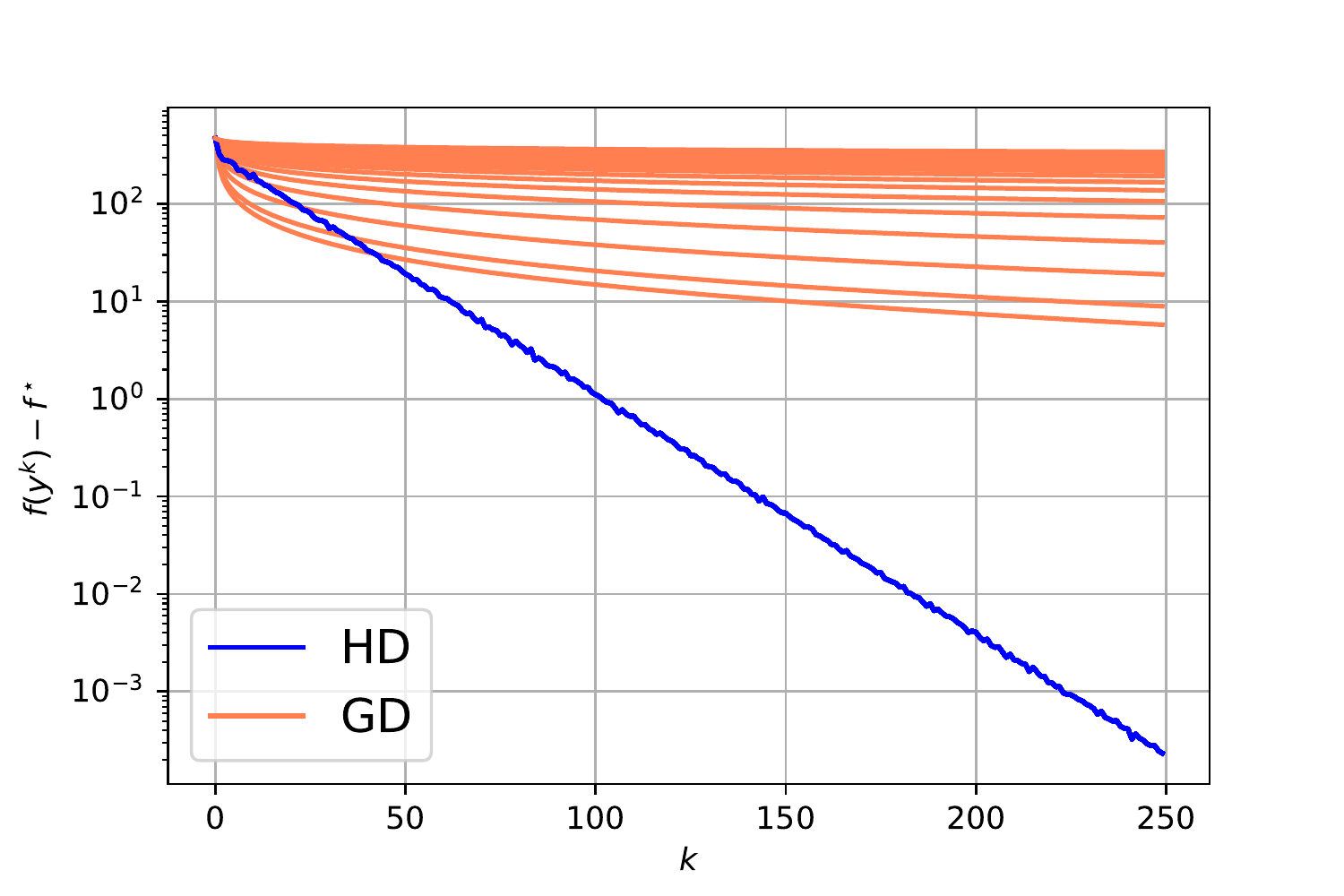}
\caption{Primal objective value.}
\label{f-quad_primal}
\end{subfigure}
\begin{subfigure}{0.47\textwidth}
\centering
\includegraphics[width=\linewidth]{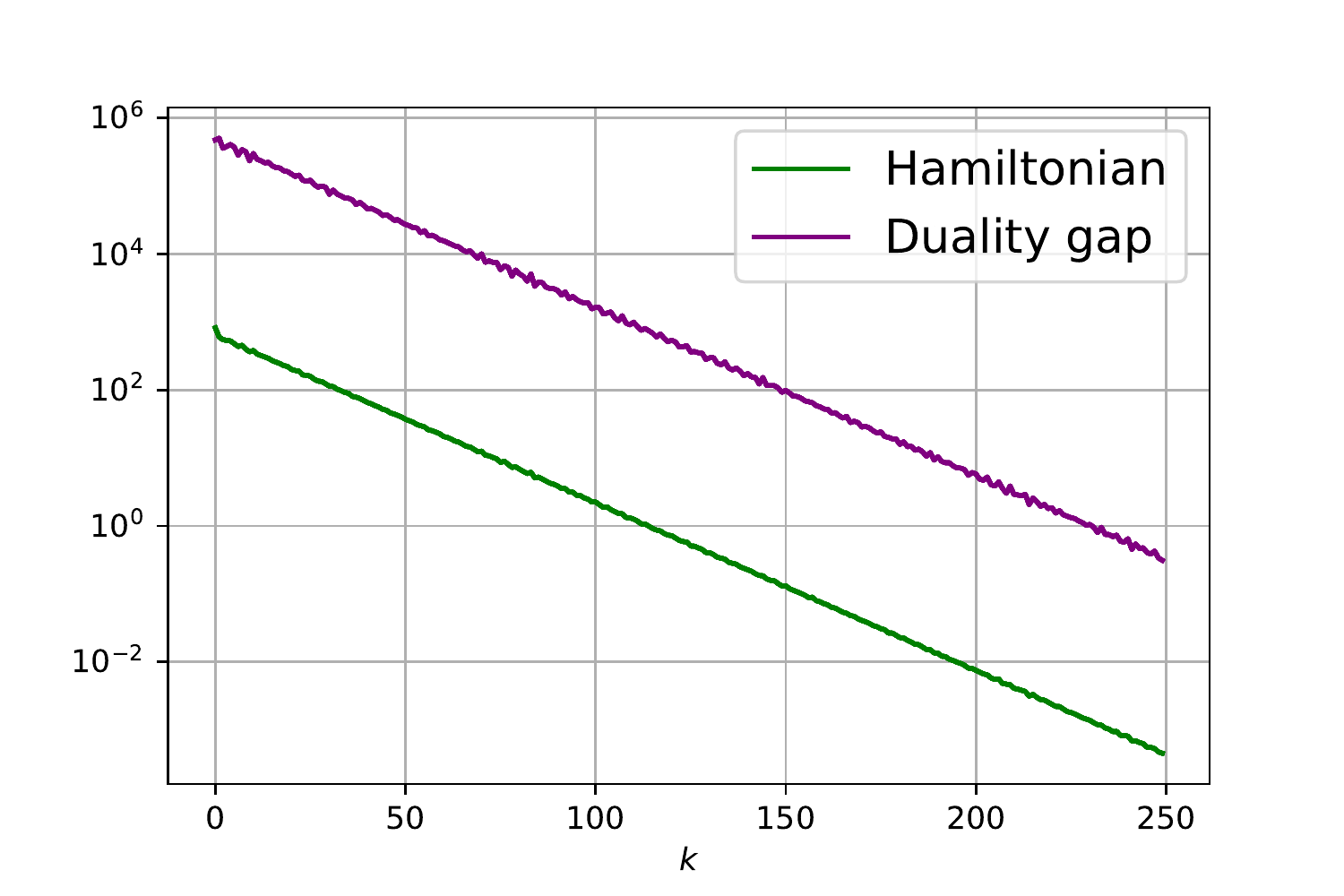}
\caption{Hamiltonian value and duality gap for HD.}
\label{f-ham-gap}
\end{subfigure}
\caption{Comparison of Hamiltonian descent (HD) and Gradient descent (GD) for
problem (\ref{e-ls}).} \label{f-quad}
\end{figure}

\subsection{Elastic net regularized logistic regression}
In logistic regression the goal is to learn a classifier to separate
a set of data points based on their labels, which we take to be either $1$ or
$-1$. The elastic net is a type of regularization that promotes sparsity and
small weights in the solution \cite{zou2005regularization}. Given data points
$a_i \in \reals^m$ with corresponding label $l_i \in \{-1, 1\}$ for
$i=1,\ldots,n$, the elastic net regularized logistic regression problem is given
by
\begin{equation}
\label{e-logreg}
\begin{array}{ll}
\mbox{minimize} & (1/n)\sum_{i=1}^n \log(1 + \exp(l_i a_i^T y)) + \lambda_1\|y\|_1 + (\lambda_2/2) \|y\|_2^2
\end{array}
\end{equation}
over the variable $y \in \reals^m$, where $\lambda_1 \geq 0$, and $\lambda_2 \geq 0$
control the strength of the regularization.
In the notation of problem (\ref{e-primal}) we take $h(x) = (1/n) \sum_{i=1}^n
\log(1 + \exp(l_i x_i))$ and $g(y) = \lambda_1\|y\|_1 + (\lambda_2/2)\|y\|_2^2$.
We have a closed form expression for the gradient of $g^*$ given by
the soft-thresholding operator:
\[
(\nabla g^*(q))_i = (1/\lambda_2) \left\{
\begin{array}{ll}
q_i - \lambda_1 & q_i \geq \lambda_1 \\
0 & |q_i| \leq \lambda_1 \\
q_i + \lambda_1 & q_i \leq -\lambda_1.
\end{array}
\right.
\]
We compare the explicit discretization (\ref{e-explicit}) of Hamiltonian descent
in equation (\ref{e-dyn2}) to proximal gradient descent \cite{parikh2014proximal},
which in this case has the exact same per-iteration cost since it also relies
on taking the gradient of $h$ and applying the soft-thresholding operator.
We chose dimension $m=500$ and $n=1000$ data points and we set $\lambda_1 =
\lambda_2 = 0.01$.  The data were generated
randomly, and then perturbed so as to give a high condition number, which was
$1.0 \times 10^8$.  The best step size for both algorithms was found using
exhaustive search.  In Figure \ref{f-log_reg_ham_v_gd} we show the primal
objective value error for both algorithms, where the true solution was
found using convex cone solver SCS \cite{ocpb:16, scs}.  Hamiltonian descent
dramatically outperforms gradient descent on this problem, despite
having the same per-iteration cost. This is unsurprising because we would expect
Hamiltonian descent to be less sensitive to the poor conditioning of the data,
due to the affine invariance property.

\begin{figure}
\centering
\includegraphics[width=0.5\linewidth]{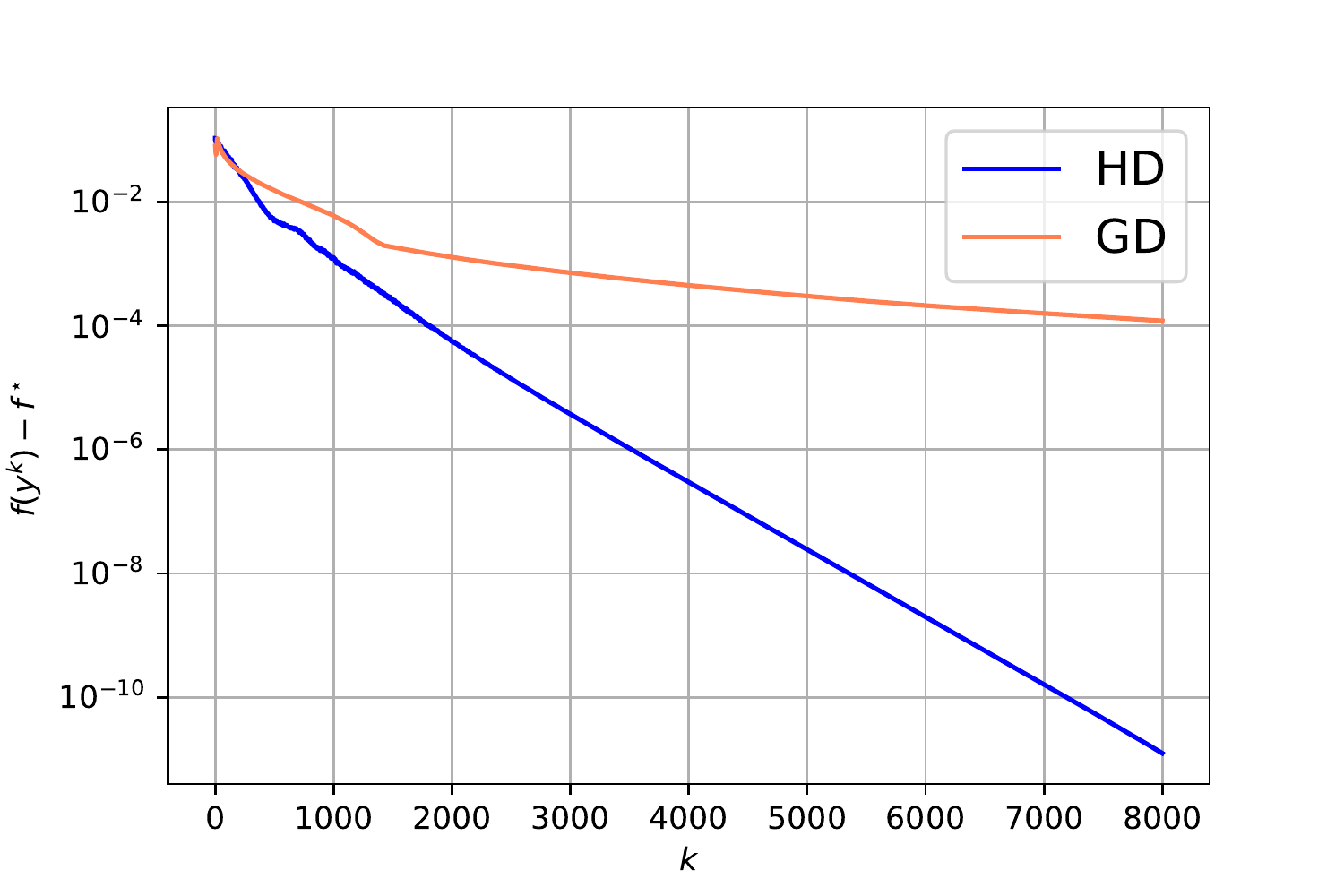}
\caption{Comparison of Hamiltonian descent (HD) and Gradient descent (GD) for
problem (\ref{e-logreg}).}
\label{f-log_reg_ham_v_gd}
\end{figure}

\section{Conclusion} Starting from Hamiltonian mechanics in classical physics,
we derived a Hamiltonian descent continuous ODE that converges linearly to a
minimum of the Hamiltonian function. We applied Hamiltonian descent to a convex
composite optimization problem, and proved linear convergence of the duality
gap, a measure of how far from optimal a primal-dual point is. In some sense
applying Hamiltonian descent to this problem is natural, since we can identify
one of the terms in the objective as being the `potential' energy and the other
as the `kinetic' energy.  We provided two discretizations that are guaranteed to
converge to the optimum under certain assumptions, and also demonstrated that
some well-known algorithms correspond to other discretizations of our ODE. In
particular we show that a particular discretization yields ADMM. We conclude
with two numerical examples that show our method is much more robust to
numerical conditioning than standard gradient methods.


\bibliographystyle{unsrt}
\bibliography{hamiltonian}
\newpage
\section*{Appendix}
\subsection*{Necessary and sufficient conditions for optimality of (\ref{e-pd})}
Recall that the primal-dual problems we are considering are:
\begin{align*}
  \begin{array}{ll}
    \mbox{minimize} & h(x) + g(y)\\
    \mbox{subject to} & x = Ay,
  \end{array}
  &
  \quad
  \begin{array}{ll}
    \mbox{maximize} & -h^*(-p) - g^*(q)\\
    \mbox{subject to} & q = A^Tp,
  \end{array}
\end{align*}
over primal variables $y \in \reals^m$, $x \in \reals^n$, dual variables
$p \in \reals^n$, $q \in \reals^m$, where matrix $A \in \reals^{n \times m}$ is
data and the functions $h : \reals^n \rightarrow \reals \cup \{\infty\}$ and $g
: \reals^m \rightarrow \reals \cup \{ \infty \}$ are convex, with convex
conjugates $h^*$ and $g^*$ respectively. The necessary and sufficient condition
for optimality of $y_\star$ for the primal problem is
\[
0 \in A^T \partial h (Ay_\star) + \partial g(y_\star).
\]
We can rewrite the optimality condition as, $\exists p_\star \in - \partial h(A y_\star)$,
\begin{align}
\begin{split}
\label{e-app-kkt1}
0 &\in -A^Tp_\star + \partial g(y_\star).
\end{split}
\end{align}
Note we have the following property for any proper convex
$f$ (see Theorem 23.5 of \cite{rockafellar1970convex}).
For any $x^*, x \in \reals^m$,
\begin{align}
x \in \partial f^*(x^*) \qquad \text{iff} \qquad x^* \in \partial f(x)
\end{align}
Now using this fact, for any $p_\star$ satisfying \eqref{e-app-kkt1} we get the following two inclusions
\begin{align}
\begin{split}
Ay_\star &\in A\partial g^*(A^Tp_\star)\\
Ay_\star &\in \partial h^*(-p_\star)\\
\end{split}
\end{align}
Together, this implies
\[
0 \in -\partial h^*(-p_\star) + A\partial g^*(A^T p_\star),
\]
which is the necessary and sufficient condition for $p_\star$ to be optimal for
the dual problem. Then taking primal-dual optimal $(y_\star, p_\star)$ and introducing $x_\star = A y_\star$ and $q_\star = A^T p_\star$,
we get
\begin{align*}
 y_\star &\in \partial g^*(q_\star)\\
 x_\star &= Ay_\star\\
 -p_\star &\in \partial h(x_\star)\\
 q_\star &= A^Tp_\star,
\end{align*}
which are the necessary and sufficient conditions for $(x_\star, y_\star,
p_\star, q_\star)$ to be primal-dual optimal for (\ref{e-pd}). In the main
text we assumed that $h$ and $g^*$ were differentiable, in which case we
can replace subdifferentials with gradients, and inclusion
with equality.

\subsection*{Relationship between the duality gap and Bregman divergences}
Starting with the definition of Bregman divergences and noting that $\nabla
h(x_\star) = -p_\star$,
\begin{align*}
D_h(x, x_\star)
&=h(x) - h(x_\star) + p_\star^T(x - x_\star),
\end{align*}
and similarly using $\nabla g^*(q_\star) = y_\star$,
\begin{align*}
D_{g^*}(q, q_\star)
&=g^*(q) - g^*(q_\star) - y_\star^T(q - q_\star).
\end{align*}
Using $p_\star^T x_\star = p_\star^T(A y_\star) = q_\star^Ty_\star$
and summing the two Bregman divergences yields the gap (\ref{e-dual-gap}).

Now let us define
\begin{align*}
\hat D_{h^*}(-p,-p_\star)
&=h^*(-p) - h^*(-p_\star) - x_\star^T(-p + p_\star),
\end{align*}
and
\begin{align*}
\hat D_{g}(y,y_\star)
&=g(y) - g(y_\star) - q_\star^T(y - y_\star),
\end{align*}
which are both nonnegative due to the convexity of $h^*$ and $g$,
and note that if $h^*$ and $g$ are differentiable then these are just
Bregman divergences, in which case we could drop the `hat' notation.

Now we shall show that the usual duality gap can be decomposed into the sum of
four (pseudo)-Bregman divergences.
Let $\hat D_f(y, y_\star) = f(y) - f(y_\star)$, which by the linearity of
(pseudo)-Bregman divergences satisifes
\[
\hat D_f(y, y_\star) = \hat D_{h \circ A + g}(y, y_\star) = D_h(Ay, x_\star) + \hat D_g(y, y_\star)
\]
and similarly denoting $\hat D_d(p, p_\star) = -d(p) + d(p_\star)$ we have
\[
\hat D_d(p, p_\star) = \hat D_{h^*}(-p, -p_\star) + D_{g^*}(A^Tp, q_\star).
\]
Summing these and using the fact that strong duality implies that $f(y_\star) = d(p_\star)$ we obtain
\[
f(y) - d(p) =  D_h(Ay, x_\star) + \hat D_g(y, y_\star) + \hat D_{h^*}(-p, -p_\star) + D_{g^*}(A^Tp, q_\star),
\]
which, due to the nonnegativity of $\hat D_{h^*}$ and $\hat D_g$, implies that
\[
f(y) - d(p) \geq D_h(Ay, x_\star) + D_{g^*}(A^Tp, q_\star) = \mathrm{gap}(Ay, A^Tp).
\]

\subsection*{Proof of generalized Moreau decomposition}
\begin{lemma}
Given a convex, closed, proper function $f : \reals^m
\rightarrow \reals$, matrix $A\in\reals^{n \times m}$, and $\rho > 0$. Any point
$x \in \reals^n$ satisfies
\[
x = (I + \rho A \partial f A^T)^{-1} x + \rho A(\partial f^* + \rho A^TA)^{-1} A^Tx.
\]
\end{lemma}
\begin{proof}
Recall that $(I + \rho A \partial f A^T)^{-1}$
is always single-valued, because it is the proximal operator of the convex
function $f \circ A^T$ \cite{parikh2014proximal}.  So to start we shall show that
$A(\partial f^* + \rho A^T A)^{-1} q$ is also single-valued for any $q$, $A$,
and convex $f^*$. Choose $y$ and $z$ to be any two elements of $(\partial f^* +
\rho A^TA)^{-1} q$. We shall show that it must be the case that $Ay = Az$, even if $z
\neq y$. Membership of the set implies that
\begin{align*}
q - \rho A^TAy &\in \partial f^*(y)\\
q - \rho A^TAz &\in \partial f^*(z),
\end{align*}
therefore by convexity and the definition of subdifferentials we have
\begin{align*}
f^*(z) &\geq f^*(y) + (q - \rho A^TAy)^T(z - y)\\
f^*(y) &\geq f^*(z) + (q - \rho A^TAz)^T(y - z),
\end{align*}
and adding these we get
\begin{align*}
0 &\geq (q - \rho A^TAy)^T(z - y) + (q - \rho A^TAz)^T(y - z)\\
&=\rho (A^TAy)^T(y-z) - \rho (A^TAz)^T(y-z)\\
&= \rho (y - z)^TA^TA(y-z)\\
&=\rho \|A(y-z)\|_2^2,
\end{align*}
which implies that $Az = Ay$, so
$A(\partial f^* + \rho A^TA)^{-1} q$ must be single-valued.

Now let $y = (\partial f^* + \rho A^T A)^{-1} A^Tx$ (which is valid, because it
is single-valued). We will make use of the following fact for any proper convex
$f$ (see Theorem 23.5 of \cite{rockafellar1970convex}).
For any $x^*, x \in \reals^m$,
\begin{align}
\label{eq:inclusions}
x \in \partial f^*(x^*) \qquad \text{iff} \qquad x^* \in \partial f(x)
\end{align}
Now using \eqref{eq:inclusions},
\begin{align*}
A^T(x - \rho Ay) \in \partial f^*(y) &\implies y \in (\partial f A^T)(x - \rho Ay) \\
&\implies \rho Ay \in \rho (A \partial f A^T)(x - \rho Ay)\\
&\implies x \in (I + \rho A \partial f A^T)(x - \rho Ay)
\end{align*}
Since $(I + \rho A \partial f A^T)^{-1}x$ is single valued, we can use
\eqref{eq:inclusions} along with the definition of $y$ to finish the proof:
\begin{align*}
x &= (I + \rho A \partial f A^T)^{-1} x + \rho Ay\\
	&= (I + \rho A \partial f A^T)^{-1} x + \rho A(\partial f^* + \rho A^TA)^{-1} A^Tx.
\end{align*}
\end{proof}
The Moreau decomposition can be seen as a
generalization of an orthogonal decomposition induced by a subspace, and
the standard statement of the theorem assumes that $A = I$, see, \eg,
\cite{parikh2014proximal}.  This extension can be interpreted as a decomposition
when the projection is weighted by the matrix $A$, since
\begin{align*}
  \argmin_v \left(f(A^Tv) + (1/2)\|v - y\|_2^2\right) &= (I +   A\partial fA^T)^{-1}y\\
  \argmin_u \left(f^*(u) + (1 / 2)\|Au - x \|_2^2\right) &=(\partial f^* + A^TA)^{-1} A^T x.
\end{align*}

\newcommand{\grad}{\nabla}
\subsection*{Convergence of Explicit discretization scheme when $\grad \Hc$ is $L$-Lipschitz}
To show convergence of the scheme presented in equation (\ref{e-explicit}) we
shall use the additional assumption that $\Hc$ has an $L$-Lipschitz gradient,
which implies that
\begin{align*}
\Hc(v) &\geq \Hc(u) + \nabla \Hc(u)^T (v-u) + (1/2L) \| \nabla \Hc(v) - \nabla \Hc(u) \|_2^2 \\
\Hc(v) &\leq \Hc(u) + \nabla \Hc(u)^T (v-u) + (L/2) \| v - u \|_2^2,
\end{align*}
for any $u, v$. Using this we can write:
\begin{align}
\label{e-app-explicit-1}
\begin{split}
  \Hc(z^{k+1}) - \Hc(z^k) &\leq \nabla \Hc(z^{k})^T(z^{k+1} - z^k) +(L/2) \|z^{k+1} - z^k \|_2^2\\
  &=\epsilon\nabla \Hc(z^{k})^T(J \nabla \Hc(z^{k}) + z_\star - z^{k})  +(\epsilon^2 L/2) \|J \nabla \Hc(z^{k}) + z_\star - z^{k}\|_2^2\\
  &\leq -\epsilon \Hc(z^k) - (\epsilon/2L) \|\nabla \Hc(z^k)\|_2^2  +(\epsilon^2
  L/2) \|J \nabla \Hc(z^{k}) + z_\star - z^{k}\|_2^2,
\end{split}
\end{align}
where the first inequality is a consequence of the Lipschitz assumption,
and the last is a combination of the Lipschitz assumption and the fact that
$J$ is skew symmetric.
Now we will use the following identity:
\[
\|(1-\epsilon)u + \epsilon v\|_2^2 = (1-\epsilon)\|u\|_2^2 + \epsilon\|v\|_2^2 - \epsilon(1-\epsilon)\|u - v\|_2^2
\]
for any $u,v$ and $\epsilon \in \reals$. We apply this to the following
\[
\|z^{k+1} - z_\star\|_2^2 = (1-\epsilon)\|z^k - z_\star\|_2^2 + \epsilon \|\nabla \Hc(z^k)\|_2^2 - \epsilon(1-\epsilon)\|J\nabla \Hc(z^k) + z_\star - z^k\|_2^2
\]
where we used the fact that $\|J \nabla \Hc(z)\|_2^2 = \|\nabla \Hc(z)\|_2^2$
since $J^TJ = I$. This allows us to replace the last term in (\ref{e-app-explicit-1})
\begin{align}
\label{e-app-explicit-2}
\begin{split}
  \Hc(z^{k+1}) - \Hc(z^k) &\leq -\epsilon \Hc(z^k) - (\epsilon/2L) \|\nabla \Hc(z^k)\|_2^2  +\\
  &\quad\frac{\epsilon L}{2(1-\epsilon)} \left((1-\epsilon)\|z^k - z_\star\|_2^2 + \epsilon \|\nabla \Hc(z^k)\|_2^2 - \|z^{k+1} - z_\star\|_2^2\right).
\end{split}
\end{align}
Now select $\epsilon$ to satisfy
\[
\frac{\epsilon}{2L} \geq \frac{\epsilon^2 L}{2(1-\epsilon)},
\]
which removes the terms involving $\|\nabla \Hc(z^k)\|_2^2$. For simplicity
we shall take $\epsilon = 1 / (L^{2} + 1)$, which satisfies the condition.
Now we take the sum of (\ref{e-app-explicit-2}), which telescopes to yield
\begin{equation}
\label{e-app-explicit-3}
  \Hc(z^T) - \Hc(z^0)
  \leq -\epsilon \sum_{k=0}^{T-1}\Hc(z^k) +
  (2L)^{-1}(\|z^{0} - z_\star\|_2^2 - \|z^{T} - z_\star\|_2^2).
\end{equation}
Now consider the averaged iterate $\bar z^T = (1/T) \sum_{k=0}^{T-1} z^k$
\[
  \Hc(\bar z^T) \leq \frac{1}{T}
  \sum_{k=0}^{T-1}\Hc(z^k) \leq \frac{1}{\epsilon T}\left(\Hc(z^0) + (2L)^{-1}\|z^{0} - z_\star\|_2^2 \right),
\]
where the first inequality is Jensen's, and the second follows from (\ref{e-app-explicit-3})
and the nonnegativity of $\Hc$. In other words $\Hc(\bar z^k) \rightarrow 0$,
and the rate of convergence is $O(1/k)$.

\subsection*{Convergence of Explicit discretization scheme when $\grad \Hc$ is $L$-Lipschitz and $\Hc$ is $\mu$ strongly convex}
Here we show the convergence of the scheme presented in equation (\ref{e-explicit}) under the assumption that $\Hc$ has an $L$-Lipschitz gradient and is $\mu$ strongly convex for $L \geq \mu > 0$. We show that the $\Hc(z^k)$ converges linearly.

The assumption of $L$-Lipschitz gradients implies,
\begin{align*}
\Hc(v) &\geq \Hc(u) + \nabla \Hc(u)^T (v-u) + (1/L) \| \nabla \Hc(v) - \nabla \Hc(u) \|_2^2/2 \\
\Hc(v) &\leq \Hc(u) + \nabla \Hc(u)^T (v-u) + L \| v - u \|_2^2/2,
\end{align*}
for any $u, v$. The $\mu$ strong convexity assumption implies
\begin{align*}
\Hc(v) &\leq \Hc(u) + \nabla \Hc(u)^T (v-u) + (1/\mu) \| \nabla \Hc(v) - \nabla \Hc(u) \|_2^2/2 \\
\Hc(v) &\geq \Hc(u) + \nabla \Hc(u)^T (v-u) + (\mu) \| v - u \|_2^2/2,
\end{align*}
for any $u, v$. In particular, we use
\begin{align*}
\mu \Hc(z) &\leq   \| \nabla \Hc(z) \|_2^2/2 \leq L \Hc(z) \\
\mu \| z - z_\star \|_2^2/2 &\leq \Hc(z) \leq  L \| z - z_\star \|_2^2/2
\end{align*}
Using this we can write:
\begin{align}
\label{e-app-explicit-1-strngcvx}
\begin{split}
  \Hc(z^{k+1}) - \Hc(z^k) &\leq \nabla \Hc(z^{k})^T(z^{k+1} - z^k) +(L/2) \|z^{k+1} - z^k \|_2^2\\
  &=\epsilon\nabla \Hc(z^{k})^T(J \nabla \Hc(z^{k}) + z_\star - z^{k})  +(\epsilon^2 L/2) \|J \nabla \Hc(z^{k}) + z_\star - z^{k}\|_2^2\\
  &\leq -\epsilon \Hc(z^k) - (\epsilon/2L) \|\nabla \Hc(z^k)\|_2^2  +(\epsilon^2
  L/2) \|J \nabla \Hc(z^{k}) + z_\star - z^{k}\|_2^2,\\
  &\leq -\epsilon \left(1 + \frac{\mu}{L}\right) \Hc(z^k) +(\epsilon^2
  L/2) \|J \nabla \Hc(z^{k}) + z_\star - z^{k}\|_2^2,
\end{split}
\end{align}
where the first inequality is a consequence of the Lipschitz assumption,
the second is a combination of the Lipschitz assumption and the fact that
$J$ is skew symmetric.
Now using triangle and Jensen's inequalities:
\begin{equation}
\Hc(z^{k+1}) - \Hc(z^k)\leq -\epsilon \left(1 + \frac{\mu}{L}\right) \Hc(z^k) + \epsilon^2
  L \left(\|\nabla \Hc(z^{k})\|_2^2 + \|z_\star - z^{k}\|_2^2\right),
\end{equation}
where we used the fact that $\|J \nabla \Hc(z)\|_2^2 = \|\nabla \Hc(z)\|_2^2$.
All together, we have
\begin{equation}
\Hc(z^{k+1}) - \Hc(z^k)\leq -\epsilon \left(1 + \frac{\mu}{L}\right) \Hc(z^k) + 2 \epsilon^2
  L^2 \Hc(z^k) + 2 \epsilon^2 \frac{L}{\mu} \Hc(z^k),
\end{equation}
Thus, if $2\epsilon \leq (L^2 +  L / \mu)^{-1}$, we have
\begin{equation}
\Hc(z^{k+1}) \leq \left(1 - \epsilon \frac{\mu}{L} \right) \Hc(z^k) \leq \left(1 - \epsilon \frac{\mu}{L} \right)^k \Hc(z^0)
\end{equation}
Taking $2 \epsilon = (L^2 +  L / \mu)^{-1}$ for simplicity we have
\begin{equation}
\Hc(z^{k+1}) \leq \left(1 - \frac{\mu}{2L^2\mu + 2L} \frac{\mu}{L} \right)^k \Hc(z^0)
\end{equation}

\subsection*{PDHG corresponds to a discretization of Hamiltonian descent}
The Hamiltonian descent equations are given by
\begin{align*}
\dot y_t &= \nabla g^*(q_t) - y_t\\
\dot q_t &= -A^T\nabla h(Ay_t) - q_t,
\end{align*}
and if we parameterize $q_t = A^T p_t$ then we can rewrite these as
\begin{align*}
\dot y_t &= \nabla g^*(A^Tp_t) - y_t\\
\dot p_t &= -\nabla h(Ay_t) - p_t.
\end{align*}
Now we use the same trick as before, introducing identical terms that we add and
subtract
\begin{align*}
\dot y_t &= \nabla g^*(A^Tp_t + y_t / \sigma - y_t /\sigma) - y_t\\
\dot p_t &= -\nabla h(Ay_t + p_t / \rho- p_t /\rho) - p_t,
\end{align*}
and then discretize as follows (which is valid due to the fact that we assumed that the
Hamiltonian was continuously differentiable):
\begin{align*}
(p^{k+\epsilon} - p^k)/\epsilon  &= -\nabla h(Ay^{k} + p^{k+\epsilon} / \rho- p^k /\rho) - p^k\\
(y^{k+\epsilon} - y^k)/\epsilon &= \nabla g^*(A^Tp^{k+\epsilon} + y^k / \sigma - y^{k+\epsilon} /\sigma) - y^k.
\end{align*}
Setting $\epsilon = 1$ and rearranging yields
\begin{align*}
\partial h^*(-p^{k+1})- p^{k+1} /\rho  &=Ay^{k} - p^k / \rho\\
\partial g(y^{k+1}) + y^{k+1} /\sigma &= A^Tp^{k+1} + y^k / \sigma,
\end{align*}
and finally
\begin{align*}
p^{k+1} &=-(I + \rho \partial h^*)^{-1}(\rho Ay^{k} - p^k)\\
y^{k+1} &= (I + \sigma \partial g)^{-1}(\sigma A^Tp^{k+1} + y^k),
\end{align*}
which is PDHG.

\subsection*{Other gradient methods on problem (\ref{e-ls})}
\begin{figure}[ht]
\centering
\begin{subfigure}{0.47\textwidth}
\centering
\includegraphics[width=\linewidth]{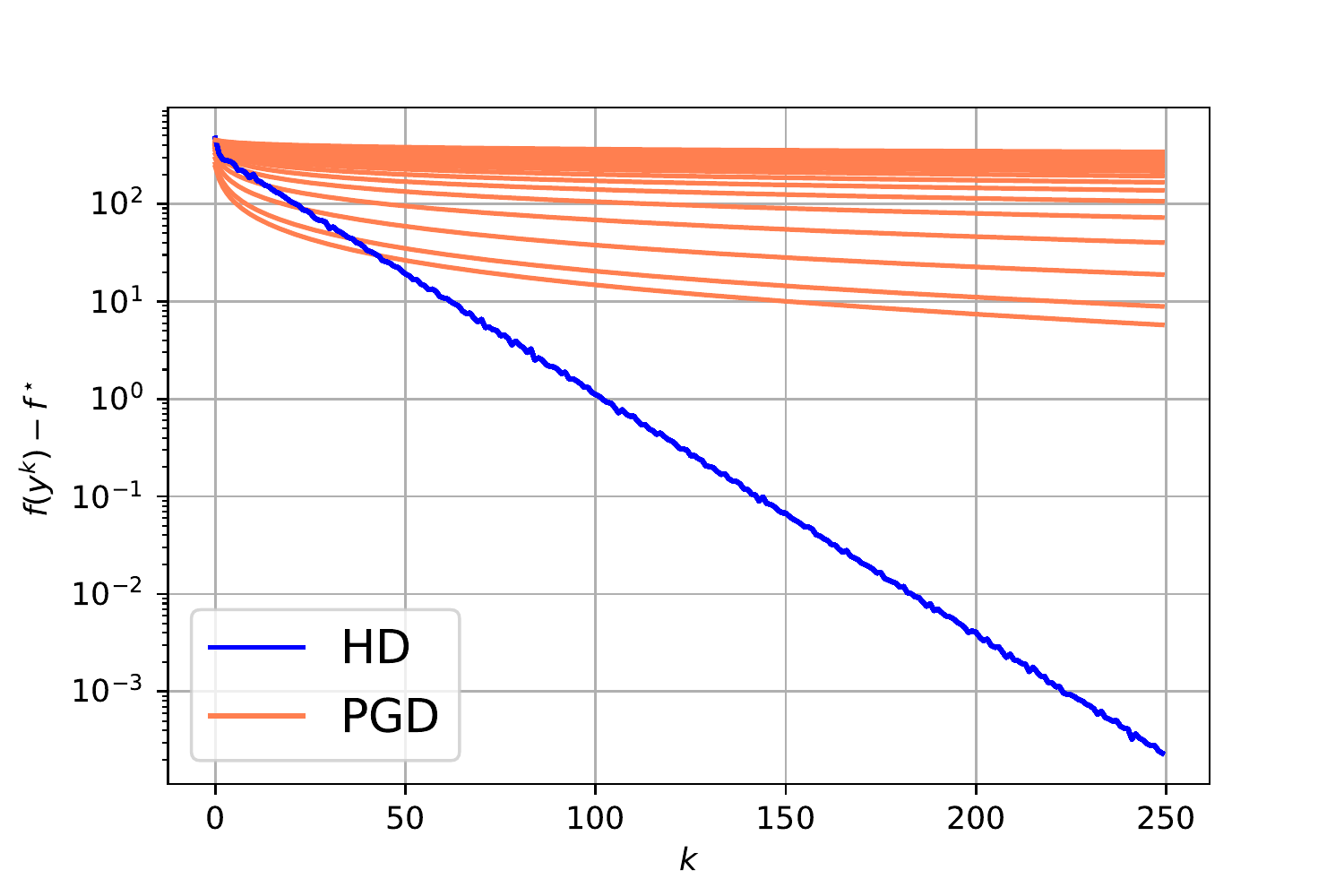}
\caption{HD and proximal gradient descent (PGD).}
\end{subfigure}
\begin{subfigure}{0.47\textwidth}
\centering
\includegraphics[width=\linewidth]{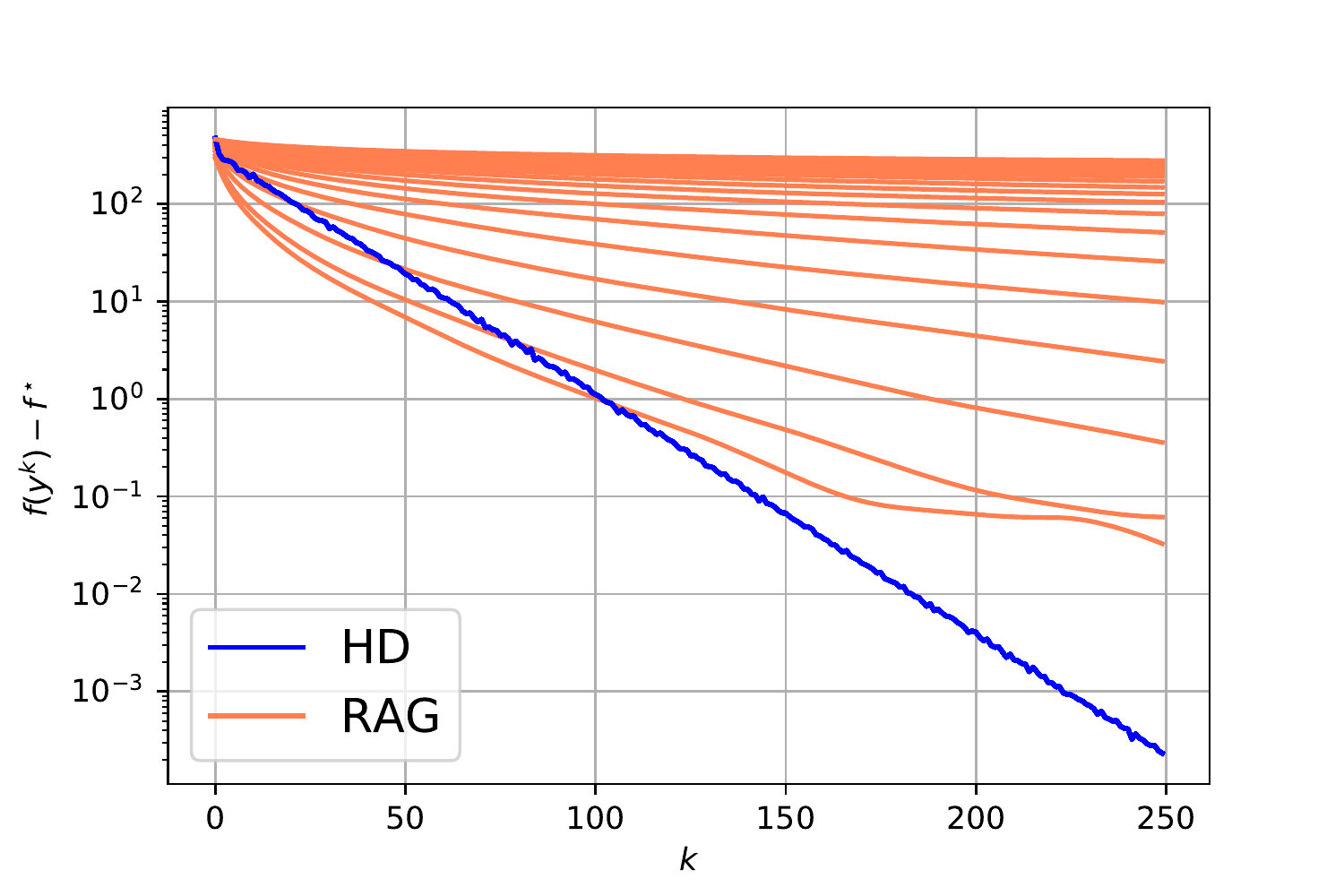}
\caption{HD and restarted accelerated gradient (RAG).}
\end{subfigure}
\begin{subfigure}{0.47\textwidth}
\centering
\includegraphics[width=\linewidth]{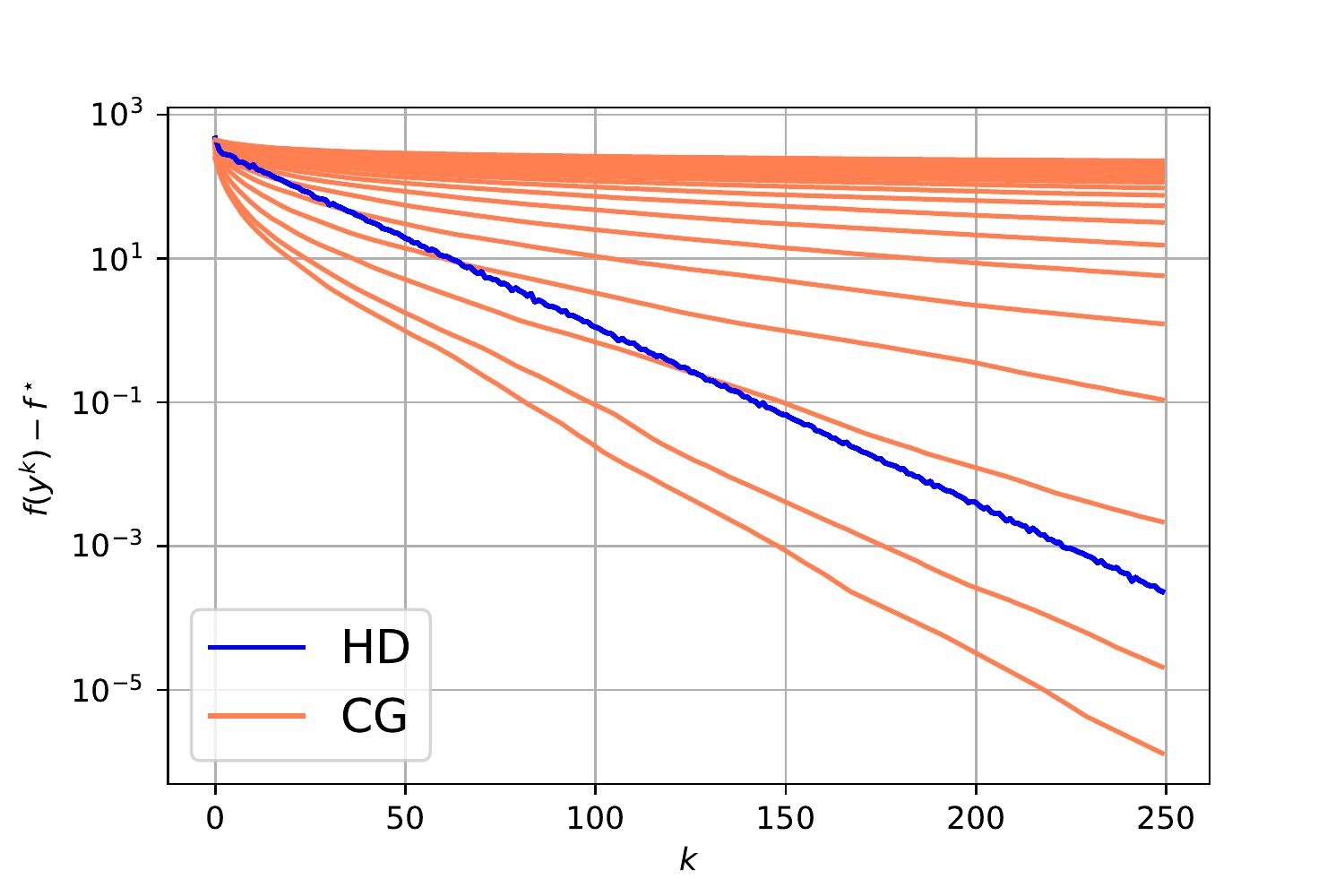}
\caption{HD and conjugate gradient (CG).}
\end{subfigure}
\caption{Comparison of Hamiltonian descent (HD) and other gradient methods for
problem (\ref{e-ls}) for different $j$.}
\label{f-other_grads}
\end{figure}
\end{document}